\documentclass[12pt,bezier]{article}
\usepackage{amsmath,amsthm,amssymb,amsfonts,euscript,graphicx}

\newtheorem{thm}{Theorem}[section]

\newtheorem{cor}[thm]{Corollary}
\theoremstyle{definition}
\newtheorem{definition}[thm]{Definition}
\newtheorem{example}[thm]{Example}

\theoremstyle{remark}
\newtheorem{remark}[thm]{Remark}

\title{\large \bf Acyclicity for Groups and Vector Spaces
}

\author{{\normalsize{\sc M. ALIABADI${}^{\mathsf{a,1}}$},{\sc H. JOLANY${}^{\mathsf{b,2}}$},{\sc M. AMIN KHAJEHNEJAD${}^{\mathsf{c,3}}$},}\\
\normalsize{{\sc M. J. MOGHADDAMZADEH${}^{\mathsf{d,4}}$},{\sc H. SHAHMOHAMAD${}^{\mathsf{e,5}}$}}\\
\vspace{3mm}\\
{\footnotesize{${}^{\mathsf{a}}$\it Department of Mathematics, Statistics, and Computer Science, University of Illinois, }}\\[-3mm]
 {\footnotesize{851 S. Morgan St, Chicago, IL 60607, USA}}\\
{\footnotesize{${}^{\mathsf{b}}$\it Universit\'e des Sciences et Technologic de Lille UFR de Math\'ematiques, Laboratoire Paul Painlev\'e,}}\\[-3mm]
{\footnotesize{ CNRS-UMR 8524 59655 Villeneuve d'Ascq Cedex/France}}\\
{\footnotesize{${}^{\mathsf{c}}$\it Department of  Electrical Engineering, Califonia Institute of Technology, Pasadena CA 91125}}\\
{\footnotesize{${}^{\mathsf{d}}$\it Department of Mathematical Sciences, Sharif University of Technology,  P. O. Box 11365-9415, }}\\[-3mm]
{\footnotesize{Tehran, IRAN}}\\
{\footnotesize{${}^{\mathsf{e}}$\it School of Mathematical Sciences, RIT, Rochester, NY 14623, USA}}\\
{\footnotesize{$^1$E-mail address: $\mathsf{maliab2@uic.edu}$}}\\
{\footnotesize{$^2$E-mail address: $\mathsf{hassan.jolany@math.univ-lille1.fr}$}}\\
{\footnotesize{$^3$E-mail address: $\mathsf{amin@caltech.edu}$}}\\
{\footnotesize{$^4$E-mail address: $\mathsf{javad\_mz123@yahoo.com}$}}\\
{\footnotesize{$^5$E-mail address: $\mathsf{hxssma@rit.edu}$}}}

\date{}
\begin{document}

\maketitle
\begin{abstract}
\noindent
A \textit{matching} in an Abelian group $G$ is a bijection $f$ from a subset $A$ to a subset $B$ in $G$ such that $a+f(a) \not\in A$, for all $a\in A$. This notion was introduced by Fan and Losonczy who used matchings in $\mathbb{Z}^n$ as a tool for studying an old problem of Wakeford concerning canonical forms for symmetric tensors. The notion of  \textit{acyclic matching property} was provided by Losonczy and it was proved that torsion-free groups admit this property.  In this paper, we introduce a duality of acyclic matching  as a tool for classification of some Abelian groups; moreover,  we study matchings for vector spaces and give a connection between matchings in groups  and vector spaces. Our tools mix additive number theory, combinatorics and algebra.
\end{abstract}
\vspace{9mm}
\section{Introduction}
Let $G$ be a group and $A$ and $B$ be two non-empty subsets of $G$. If $f: A\to B$ is a matching, we define $m_f: G\to \mathbb{Z}\cup \{\infty\}$ by $m_f(x)=\#\{a\in A:\; a+f(a)=x\}$. A matching  $f$ is called  acyclic if for any matching $g: A\to B$ with $m_f=m_g$, we have $f=g$. A group $G$ possesses the \textit{finite matching property}
if for every pair $A$ and $B$ of non-empty finite subsets satisfying $\# A=\# B$ and $0\not\in B$, there exists at least one matching from $A$ to $B$. Furthermore, $G$ possesses the \textit{finite acyclic matching property}, if for every pair $A$ and $B$ of non-empty finite subsets satisfying $\# A=\# B$ and $0\not\in B$, there exists at least one acyclic matching from $A$ to $B$. We say $G$ \textit{fails to have the acyclic matching property at order} $m\in \mathbb{N}\cup\{\infty\}$, if there exist subsets $A$ and $B$ of $G$ and matchings $f,g:A\to B$ such that $\# A=\# B=m$, $f\neq g$ and $m_f=m_g$.

Let $A$ be a subset of $\mathbb{Z}_p$ and $f: A\to A$ be a bijection, where $p$ is prime. Then $\mathrm{ord}_f(a)$ denotes the minimum positive integer $n$ for which $f^n(a)=a$, where $a\in A$. Losonczy in \cite{10} proved the following theorems:

\noindent
\begin{thm}
If $G$ is an Abelian group, then $G$ has the finite matching property if and only if $G$ is torsion-free or cyclic of prime order.
\end{thm}
\noindent
\begin{thm}
 If $G$ is an Abelian torsion-free group, then $G$ has the finite acyclic matching property.
\end{thm}
For more results on matchings see [2,4,5,6,7,8,10,11 and 13]. Also, the interested reader is referred to [12] to see more details on Wakeford's problem. Here, we prove the following theorem as a connection between acyclic matching property and its duality.

\noindent
\begin{thm}
 Let $G$ be an Abelian group and $G\neq \mathbb{Z}_2, \mathbb{Z}_3, \mathbb{Z}_5$. If $G$ has  the finite acyclic matching property, then it fails to have  the acyclic matching property at order $m$, for some $m\in \mathbb{N}\cup\{\infty\}$.
\end{thm}

\section{ Acyclic matching in a special case for some cyclic groups}
In the following theorem, we show that $\mathbb{Z}_p$ fails to have the acyclic matching property at order $\dfrac{p-1}{2}$ for $p>5$.

\noindent
\begin{thm}
 Let $p>5$ be a prime. Then $\mathbb{Z}_p$ has the cyclic matching property of order $\dfrac{p-1}{2}$.
\end{thm}
\noindent
\begin{proof}
 Choose $a$ and $b\in \{1,\ldots,p-1\}$ such that $a\neq b$, $\left(\dfrac{a}{p}\right)=
\left(\dfrac{b}{p}\right)=1$ and $\left(\dfrac{a+1}{p}\right)=\left(\dfrac{b+1}{p}\right)=-1$, where $\bigg( \, \bigg)$ denotes Legendre symbol. See \cite{3} for more results on quadratic residue modulo $p$. Set $A:=\left\{n^2 :\; n\in\mathbb{Z}_p\setminus\{0\}\}\subseteq \mathbb{Z}_p\right.$ and define the bijections $f$ and $g: A\to A$ by $f(n^2)=a n^2$ and $g(n^2)=b n^2$ for any $n\in\mathbb{Z}_p\setminus\{0\}$. Now it is clear that $f$ and $g$ are matchings with $m_f=m_g$. This follows $\mathbb{Z}_p$ fails to have the  acyclic matching property at order $\dfrac{p-1}{2}$ for $p>5$.
\end{proof}

In the next section, we generalize Theorem 2.1 without invoking the results on quadratic residue.

\section{The acyclicity in general case for some cyclic groups}
Let $A\subseteq \mathbb{Z}_p\setminus\{0\}$ and $f: A\to A$ be a bijection. If $a\in A$, then $B=\{f^i(a):\; i\in \mathbb{N}\}$ is invariant under $f$, i.e., $f(B)\subseteq B$. It is clear that there exist $ a_1,\ldots, a_n\in A$ such that $A=\left\{f^i(a_j):\; 1\leq j\leq n,\; i\in\mathbb N\right\}$. Let $A\subseteq \mathbb{Z}_p\setminus\{0\}$ and $f: A\to A$ be a matching for which $f^2\neq id_A$. There exists $a\in A$ with ord$_{f}(a)=m>2$. Now, suppose there exists $b\in A$ such that $b\not\in\left\{f^i(a):\; i\in\mathbb N\right\}$ and define $B=\left\{f^i(b):\; i\in \mathbb{N}\right\}$. Then $f\big|_{A\setminus B}: A\setminus B\to A\setminus B$ is a matching with $f\circ f\big|_{A\setminus B}\neq id_{A\setminus B}$.

 In the following theorem, we show that the torsion groups $\mathbb{Z}_p$ fail to have the acyclic matching property at order $k$, where $2<k<p-2$. It is a remarkable fact that $m_f=m^{-1}_f$, where $f$ is a matching from a non-empty subset $A$ of a group $G$ to $A$ and it is applied in the proof of the next theorem. Also, we already have seen in elementary group theory that  if the distinct cyclic representation of a permutation $\sigma \in S_n$ has a cycle with a length greater rehan 2, then $\sigma\neq \sigma^{-1}$.

\noindent
\begin{thm}
 $\mathbb Z_p$ fails to have the acyclic matching property at order $k$ for $2<k<p-2$, where $p$ is a prime greater than $5$.
\end{thm}
\begin{proof}
 First, we prove that $\mathbb Z_p$ fails to have the acyclic  matching property at order $p-3$. Set $A:=\mathbb{Z}_p\setminus\left\{0,1,p-1\right\}$ and define $f: A\to A$ by
\[
f=(4\;\; p-4)(5\;\;p-5)\cdots(\dfrac{p-1}{2}\;\;\dfrac{p+1}{2})(3\;\;p-3\;\;2\;\;p-2),
\]
where the notation $(a_1\, a_2\, \ldots\,a_n)$ denotes the permutation of the set $\{a_1,a_2,\ldots,a_n\}$ with $a_i\to a_{i+1}$ $1\leq i\leq n-1$ and $a_n\to a_1$.

Obviously, $f$ is a matching. If $g=f^{-1}$, then $f\neq g$ and $m_f=m_g$. Now, we show that  $\mathbb Z_p$ fails to have the acyclic  matching property at order $p-4$. Let $A=\mathbb{Z}_p\setminus\left\{0,4,p-4,p-1\right\}$. Define $f: A\to A$ by
\[
f=(5\;\; p-5)(6\;\;p-6)\cdots(\dfrac{p-1}{2}\;\;\dfrac{p+1}{2})(3\;\;p-3\;\;2\;\;p-2\;\;1).
\]
Thus $f$ is a matching. Assume that  $g=f^{-1}$, then $f\neq g$ and $m_f=m_g$. This yields $G$ fails to have the acyclic  matching property at order $p-3$ and $p-4$. If we remove the transpositions of the distinct cyclic representation of $f$, then $f$ still remains  a matching on the omitted subsets and if $B_i$'s are the omitted subsets, then $f\big|_{A\setminus B_i}\neq id_{A\setminus B_i}$, for every $i$, $1\leq i\leq n$. Suppose $g_i=\left(f\big|_{A\setminus B_i}\right)^{-1}$, then $f\big|_{A\setminus B_i}\neq g_i$ and $m_{f\big|_{A\setminus B_i}}=m_{g_i}$. Hence $\mathbb Z_p$ fails to have the acyclic  matching property at orders $p-3-2k$ and $p-4-2k'$ for any $1\leq k\leq\frac{p-7}{2}$ and $1\leq k'\leq \frac{p-9}{2}$. Then $\mathbb Z_p$ fails to have the acyclic  matching property at order $k$, for $3<k<p-2$. For $k=3$,  define $A=\left\{1, 2, 4\right\}$ and $f:A\to A$ by $f=( 1\;\; 2\;\; 4)$. Assuming  $g=f^{-1}$ we get the desired result.
\end{proof}

In the last theorem, we showed that $\mathbb{Z}_p$ fails to have the acyclic matching property at order $k$, where $2<k<p-2$. In the following theorem, we study its behavior at order $p-2$.

\noindent
\begin{thm}
 Let $A$ be a subset of $\mathbb Z_p\setminus\{\overline0\}$ and $\# A=p-2$. If $f: A\to A$ is a matching, then $f^2=id_{A}$.
\end{thm}
\noindent
\begin{proof}
 Let $f^2\neq id_A$ and choose $a\in A$ and the positive integer $m>2$, such that ord$_f(a)=m$. Thus $f^{i-1}(a)+f^i(a)\not\in A$, for each $i$, $1\leq i\leq m$. It is clear that $f^{i-1}(a)+f^i(a)\neq f^i(a)+f^{i+1}(a)$ for any $i$, $1\leq i\leq m$. Suppose $m$ is  even, since $\# A=p-2$ and $A\cap \left\{f^{i-1}(a)+f^i(a):\; 1\leq i\leq m\right\}=\varnothing$, then $f^{i-1}(a)+f^i(a)=f^{i+1}(a)+f^{i+2}(a)$, for any $i=1,\ldots,m-1$. Let us  $a+f(a)=n$ and $f(a)+f^2(a)=n'$, $n=a+f(a)=f^2(a)+f^3(a)=\cdots=f^{m-1}(a)+f^m(a)=n$ and ${n'}=f(a)+f^2(a)=f^3(a)+f^4(a)=\cdots=f^{m-2}(a)+f^{m-1}(a)=f^m(a)+ a$. Therefore, $\displaystyle\sum_{i=1}^mf^i(a)=(m+1)n=(m+1)n'$, so $n=n'$ and it is a contradiction. If $m$ is odd,  there exists $i$, $1\leq i\leq m$ for which $f^{i-1}(a)+f^i(a)=f^i(a)+f^{i+1}(a)$. Since $\#\left\{f^{i-1}(a)+f^i(a):\; 1\leq i\leq m\right\}\leq 2$, therefore $f^2(a)= a$,  which is a contradiction. 
\end{proof}

\noindent
\remark
 There is only one matching $f$ from $\mathbb Z_p\setminus\{0\}$ to $\mathbb Z_p\setminus\{0\}$. Then $\mathbb Z_p$ does not fail to have  the acyclic matching property at order $p-1$.

\section{Acyclicity for  Abelian torsion-free groups}
\noindent
\begin{thm}
 Let $G$ be an Abelian group. If $G$ is non-divisible and torsion-free, then it fails to have the acyclic  matching property at order $\infty$.
\end{thm}
\noindent
\begin{proof}
Assume that $n$ is the smallest positive integer such that $2nG\subsetneq G$. We break the proof in to the following cases:\\
Case 1: If $n>1$, then $G=2G$. Choose $x\in G\setminus 2ng$ and let $2nG+x=\{2ng+x:\; g\in G\}$. Define $f, g: 2nG\to 2nG+x$ by $f(2nt)=2nt+x$ and $g(2nt)=2nt+(2n+1)x$, for any $t\in G$. Since $G$ is torsion-free, $f$ and $g$ are matchings. Choose $g_0\in G$ such that $x=2g_0$ and define $A_t=\left\{y\in G:\; 4ny+x=t\right\}$ and $B_t=\left\{y\in G:\; 4ny+(2n+1)x=t\right\}$, for any $t\in G$. So, $\varphi: A_t\to B_t$ with $\varphi(y)=y-g_0$ is a bijection. Now, since $m_f(t)=\# A_t$ and $m_g(t)=\# B_t$, then $m_f=m_g$ and $G$ fails to have the acyclic  matching property.\\
Case 2: If $n=1$, choose $x\in G\setminus 2G$. The bijections $f,g: 2G\to 2G+x$ defined by $f(2t)=2t+x$ and $g(2t)=2t-3x$ are matchings. Define $A_t=\left\{y\in G:\; 4y+x=t\right\}$ and $B_t=\left\{y\in G:\; 4y-3x=t\right\}$ , for any $t\in G$. Hence $\varphi: A_t\to B_t$ by $y\mapsto y+x$ is a bijection. This yields that  $m_f=m_g$ and $G$ fails to have the acyclic  matching property at order $\infty$. 
\end{proof}

\noindent
\example
 For any integer $n$, $n\mathbb Z$ fails to have the acyclic  matching property at order $\infty$.

In the proof of the Theorem 4.4, the following result on divisible torsion-free groups will be used. See \cite{9} for more details.

\noindent
\begin{thm}
 Let $G$ be an Abelian group. If $G$ is divisible and  torsion-free, then it is a direct-sum of isomorphic copies of $\mathbb Q$.
 \end{thm}
By the aforementioned theorem, we can consider $\mathbb{Q}$ as a subset of a group $G$ under the suitable hypotheses on $G$ and we get the  following theorem:

\noindent
\begin{thm}
 Let $G$ be an Abelian group. If $G$ is divisible and torsion-free, then it fails to have the acyclic  matching property at order $\infty$.
\end{thm}
\noindent
\begin{proof}
 By Theorem 4.3, $\mathbb Q$ is embedded in $G$. Set $A:=\{2k:\; k\in\mathbb Z\}$ and $B:=\{2k+1:\; k\in \mathbb{Z}\}$ as subsets of $\mathbb Q$. Define the bijections $f, g: A\to B$ by $f(2n)=2n+1$ and $g(2n)=2n+5$. It is clear that $f$ and $g$ are matchings.  Now, if $x\in G\setminus\{4k+1:\; k\in\mathbb Z\}$, then $m_f(x)=m_g(x)=0$. On the other hand, if $x\in\{4k+1:\; k\in\mathbb Z\}$, then $m_f(x)=m_g(x)=1$ and then, in all cases $m_f=m_g$. 
\end{proof}

\noindent
\begin{cor}
By Theorem 4.1 and Theorem 4.4, if $G$ is an Abelian, torsion-free group  then $G$ fails to have the acyclic  matching property at order $\infty$.
\end{cor}
\noindent
\example
 Two additive groups $\mathbb R$ and $\mathbb Q$ fails to have the acyclic  matching property at order $\infty$.

Now, our result regarding the connection of matching properties for Abelian groups.

\noindent
\begin{thm}
 Suppose $G$ is an Abelian group and $G\neq \mathbb{Z}_2$, $\mathbb{Z}_3$ and $\mathbb{Z}_5$. If $G$ has the finite acyclic matching property, then it fails to have the acyclic matching property at order $m$ for some $m\in \mathbb{N}\cup\{\infty\}$. 
\end{thm}
We will see the proof of this theorem in  section 7.

\section{Linear version of acyclicity  for subspaces in a field extension}
Let $G$ be an Abelian group and $f$ and $g: A\to B$  be two matchings where $A$ and $B$ are non-empty finite subsets of $G$ and $m_f=m_g$. For any $x\in G$, define $A_x^f=\left\{a\in A:\; a+f(a)=x\right\}$, $A_x^g=\left\{a\in A:\; a+g(a)=x\right\}$, $\mathcal A_f=\left\{A_x^f:\; m_f(x)\neq0\right\}$ and $\mathcal A_g=\left\{A_x^g:\; m_g(x)\neq0\right\}$. It is clear that $\mathcal A_f$ and $\mathcal A_g$ are distinct decompositions for $A$ and $\#\mathcal A_f<\infty$, $\#\mathcal A_g<\infty$. Define the function $\varphi:A\to A$ by the following method:\\
Define  $\mathcal A_f=\left\{ A_{x_1}^f,A_{x_2}^f,\ldots,A_{x_m}^f\right\}$. Since $m_f=m_g$, then $\mathcal A_g=\left\{ A_{x_1}^g,A_{x_2}^g,\ldots,A_{x_m}^g\right\}$. Assume that $a_1\in A_{x_1}^f$, choose an arbitrary element $b_1$ of $A_{x_1}^g$ and put $\varphi(a_1)=b_1$. If $a_2$ is another element of $A_{x_1}^f$, choose another arbitrary element $b_2$ of $A_{x_1}^g\setminus \{b_1\}$ and put $\varphi(a_2)=b_2$.
We can continue this procedure to define $\varphi$ on $A_{x_1}^f$ and by the similar way, we can define the function $f$ on whole $A$ which is bijective and satisfies $a+f(a)=\varphi(a)+g(\varphi(a))$ for any $a\in A$.

Conversely, assume that $f$ and $g$ are two matchings from $A$ to $B$ and there exists a bijection $\varphi:A\to A$ for which $a+f(a)=\varphi(a)+g(\varphi(a))$, for any $a\in A$. We claim that $m_f=m_g$. Let  us $x$ be an arbitrary element of $G$. We have the following cases:\\
Case 1: If $x\in A$, according to the definition of matching, $m_f(x)=m_g(x)=0$.\\
Case 2: If $x\not\in A$, then $m_f(x)=\#\{a\in A:\; a+f(a)=x\}=\# \{a\in A:\; \varphi(a)+g(\varphi(a))=x\}=\#\{a\in A:\; a+g(a)=x\}=m_g(x)$. So, $m_f=m_g$, as claimed.

So we get the following theorem:

\noindent
\begin{thm}
 Let $A$, $B$, $f$ and $g$ be as above. Then, $m_f=m_g$ if and only if there exists a bijection $\varphi: A\to A$ such that $a+f(a)=\varphi(a)+g(\varphi(a))$, for any $a\in A$.
\end{thm}
By the aforementioned theorem, a natural generalization for the acyclic matching in vector spaces is inspired.  To see  this concept, we need to present some definitions from \cite{5}.

\noindent
\begin{definition}
 Let $K\subseteq L$ be a field extension and $A$, $B$ be $n$-dimensional $K$-subspaces of the field extension $L$. Let $\mathcal{A}=\{a_1,\ldots,a_n\}$ and $\mathcal B=\{b_1,\ldots,b_n\}$ be bases of $A$ and $B$, respectively. It is said $\mathcal A$ is \textit{matched} to $\mathcal B$ if
\[
a_ib\in A \;\Rightarrow \; b\in \langle b_1,\ldots,\hat{b}_i,\ldots,b_n\rangle
\]
for all $b\in B$ and $i=1,\ldots,n$, where $\langle b_1,\ldots,\hat{b}_i,\ldots,b_n\rangle$ is the hyperplane of $B$ spanned by the set $\mathcal B\setminus \{b_i\}$; moreover, it is stated that $A$ is \textit{matched } with $B$ if every basis $\mathcal A$ of $A$ can be matched to a basis $\mathcal B$ of $B$. It is said that $L$ has the \textit{linear matching property} if, for every $n\geq1$
and every $n$-dimensional subspaces $A$, $B$ of $L$ with $1\not\in B$, the subspace $A$ is matched with $B$. A \textit{strong matching} from $A$ to $B$ is a linear isomorphism $\varphi: A\to B$ such that any basis $\mathcal A$ of $A$ is matched to the basis $\varphi(\mathcal{A})$ of $B$.
\end{definition}
Now, we are in the situation  to give the linear version of acyclicity.

\noindent
\begin{definition}
 Let $K\subseteq L$ be a field extension and $A$ and $B$ be two $n$-dimensional $K$-subspaces in $L$ such that $n>1$. Let $f$, $g: A\to B$ be two strong matchings. We say $f$ is equivalent to $g$ and denote it by $f\sim g$ if there exists a linear isomorphism $\varphi: A\to A$ such that $af(a)=\varphi(a)g(\varphi(a))$, for any $a\in A$; moreover, we state that the strong matching $f: A\to B$ is \textit{linear acyclic matching} if for any strong matching $g: A\to B$, if $f\sim g$, then $f=cg$, for some $c\in K$. We say $K\subseteq L$ \textit{fails to have the linear acyclic matching} property at order $m\in \mathbb{N}$, if there exist $K$-subspaces $A$ and $B$ in $L$ and strong matchings $f$ and $g: A\to B$ such that $f\neq g$, $f\sim g$ and $\dim_KA=\dim_KB=m$.
\end{definition}
Eliahou and Lecouvey in \cite{5} proved the following theorems. The interested reader is also referred to \cite{1}.

\noindent
\begin{thm}
 Let $K\subset L$ be a field extension.  Then $K$ has the linear matching property if and only if $L$ contains no proper finite-dimensional extension over $K$.
\end{thm}
\noindent
\begin{thm}
 Let $K\subset L$ be a field extension and $A$ and $B$ be $n$-dimensional $K$-subspaces distinct from $\{0\}$. There is a strong matching from $A$ to $B$ if and only if $AB\cap A=\{0\}$. In this case, any isomorphism $\varphi: A\to B$ is a strong matching.
\end{thm}
Now, our result regarding the connection of the linear matching properties for field extensions.

\noindent
\begin{thm}
 Let $K \subsetneqq L$ be a field extension admit the linear matching property and $\# K\geq 5$. Then it fails to have the linear acyclic matching property at order $m$, for some $m\in \mathbb{N}$.
\end{thm}
We will see the proof of this theorem in section 7.

\section{The linear acyclicity of a given order}
In this section, we study  the linear acyclicity for  finite field extensions.

\noindent
\begin{thm}
 Let $K\subsetneq L$ be a field extension with $[L: K]=n$, $\# K\geq5$ and no proper intermediate subfield. Then $K\subset L$ fails to have the linear acyclic matching property at order $m$, for any $1\leq m \leq(n+1)/4$.
\end{thm}
\noindent
\begin{proof}
 Choose $m\in \mathbb N$ and $a \in L\setminus K$ for which $m\leq (n+1)/4$. Set $A_m:=\langle a, a^3, \ldots, a^{2m-1}\rangle$. Then $A_m\cap A_m^2=\{0\}$, because $K(a)=L$, for any $A\in L \setminus K$. Using Theorem 5.5, there exists a strong matching $f_m: A_m\to A_m$. Next, set $g_m:=f_m^{-1}$. One more time using Theorem 5.5, follows that $f_m^{-1}$ is a strong matching. Now, if $f_m\circ f_m\neq id_{A_m}$, then $f_m\neq g_m$ and $f_m\sim g_m$. On the other hand, if $f_m\circ f_m=id_{A_m}$, choose $c\in K$ such that $c^2\not\in\{0,1\}$. Set $h_m:=c^{-2}g_m$, then $h_m$ is a strong matching. We claim $f_m\sim h_m$. In order to prove, define $\varphi_m:=cf_m$. We get  $af_m(a)=\varphi_m(a)h_m(\varphi_m(a))$, for any $a\in A_m$. This  tells us  $f_m\sim h_m$, as claimed. 
\end{proof}

\noindent
\begin{thm}
 Let $K\subset L$ be a purely transcendental extension. Then, it fails to have the linear acyclic matching property at order $m$, for any $m\in \mathbb{N}$.
\end{thm}
\noindent
\begin{proof}
 Let $a$ be an element of $L\setminus\{0,1\}$ and set $A_m:=\langle a, a^3, \ldots, a^{2m-1}\rangle$. Then $A_m\cap A_m^2=\{0\}$ and by Theorem 5.5, there exists a  strong matching $f_m$ from $A_m$ to $A_m$. By the same method in the previous theorem we can conclude that $K\subset L$ fails to have the acyclic linear matching property at order $m$, for any $m\in\mathbb N$.
\end{proof}

\noindent
\remark
 If a  field extension $K\subseteq L$  has no finite-dimensional proper intermediate field extension and $\# K\geq5$. Then, it fails to have the acyclic matching property at order $m$, for some $m\in\mathbb N$.

\noindent
\begin{proof}
  This is a direct consequence of Theorems 6.1 and 6.2. 
\end{proof}

\section{Main results}
\noindent
\begin{thm}
 Suppose $G$ is an Abelian group and $G\neq\mathbb Z_2,\mathbb{Z}_3$ and $\mathbb{Z}_5$. If $G$ has the finite acyclic matching property, then it fails to have the acyclic matching property at order $m$ for some $m\in\mathbb N\cup\{\infty\}$.
\end{thm}
\noindent
\begin{proof}
 Assume $G$ has the finite acyclic matching property. Then $G$ has the finite matching property. Using Theorem 1.1, $G$ is cyclic of prime order or torsion-free. Invoking  Corollary 4.5 and Theorem 2.1, $G$ fails to have the acyclic matching property at order $m$ for some $m\in\mathbb N\cup\{\infty\}$.
\end{proof}

\noindent
\begin{thm}
 Let $K\subsetneqq  L$ be a field extension admit  the linear matching property and $\# K\geq5$. Then it fails to have the linear acyclic matching property at order $m$, for some $m\in \mathbb{N}$.
\end{thm}
\noindent
\begin{proof}
 If $K\subset L$ has the linear matching property, so Theorem 5.4 yields  it has no proper finite-dimensional $K$-subspaces and by Remark 6.3, $K\subset L$ fails to have the acyclic matching property at order $m$, for some $m\in \mathbb{N}$.
\end{proof}

\textbf{Acknowledgement:} The authors are thankful to Professors Noga Alon and Saieed Akbari for providing useful comments and discussions.

\end{document}